\documentclass[11pt]{amsart}

\usepackage{amssymb, amsmath, amsthm, amsfonts}

\usepackage{amsthm}
\usepackage{amssymb}

\usepackage[all]{xy}

\usepackage{verbatim}
\usepackage{hyperref}
\hypersetup{
    colorlinks=true,       
    linkcolor=blue,          
    citecolor=blue,        
    filecolor=blue,      
    urlcolor=blue           
}

\newtheorem{theorem}{Theorem}[section]
\newtheorem{cor}[theorem]{Corollary}
\newtheorem{lem}[theorem]{Lemma}
\newtheorem{prop}[theorem]{Proposition}

\theoremstyle{definition}

\newtheorem{ex}[theorem]{Example}

\theoremstyle{remark}
\newtheorem{rem}[theorem]{Remark}
\newtheorem{rems}[theorem]{Remarks}

\newcommand{\HH}{\mathrm{H}}
\newcommand{\End}{\mathrm{End}}

\newcommand{\bF}{\mathbb{F}}
\newcommand{\bP}{\mathbb{P}}

\newcommand{\bZ}{\mathbb{Z}}
\newcommand{\bQ}{\mathbb{Q}}
\newcommand{\bR}{\mathbb{R}}
\newcommand{\bC}{\mathbb{C}}
\newcommand{\bV}{\mathbb{V}}

\newcommand{\cA}{\mathcal{A}}
\newcommand{\cB}{\mathcal{B}}
\newcommand{\cC}{\mathcal{C}}
\newcommand{\cO}{\mathcal{O}}
\newcommand{\cD}{\mathcal{D}}

\newcommand{\cF}{\mathcal{F}}

\newcommand{\cL}{\mathcal{L}}
\newcommand{\cM}{\mathcal{M}}

\newcommand{\cV}{\mathcal{V}}

\newcommand{\vol}{\mathrm{vol}}
\newcommand{\arrow}{\rightarrow}
\newcommand{\rk}{\mathrm{rk}\,}

\newcommand{\Gl}{\mathrm{Gl}\,}

\newcommand{\sbt}{\,\begin{picture}(-1,1)(-1,-3)\circle*{3}\end{picture}\ }

\newcommand{\Gr}{\mathrm{Gr}}
\newcommand{\Supp}{\mathrm{Supp}}
\newcommand{\rank}{\mathrm{rank}}

\newcommand{\bD}{\mathbb{D}}
\newcommand{\Ric}{\mathrm{Ric \,}}
\newcommand{\Zero}{\mathrm{Zero \,}}
\newcommand{\K}{\mathrm{K}}

\title[Increasing hyperbolicity from VHS with level structures]{Increasing hyperbolicity of varieties supporting a variation of Hodge structures with level structures}

\author[Y. Brunebarbe]{Yohan Brunebarbe}
\address{\noindent Y. Brunebarbe:  Dept. of Mathematics, Univ. Bordeaux, Talence, France.}
\email{yohan.brunebarbe@math.u-bordeaux.fr}

\begin{document}
\begin{abstract}
Looking at the finite \'etale congruence covers $X(p)$ of a complex algebraic variety $X$ equipped with a variation of integral polarized Hodge structures whose period map is quasi-finite, we show that both the minimal gonality among all curves contained in $X(p)$ and the minimal volume among all subvarieties of $X(p)$ tend to infinity with $p$. This applies for example to Shimura varieties, moduli spaces of curves, moduli spaces of abelian varieties, moduli spaces of Calabi-Yau varieties, and can be made effective in many cases. The proof goes roughly as follows. We first prove a generalization of the Arakelov inequalities valid for any variation of Hodge structures on higher-dimensional algebraic varieties, which implies that the hyperbolicity of the subvarieties of $X$ is controlled by the positivity of a single line bundle. We then show in general that a big line bundle on a normal proper algebraic variety $\bar X$ can be made more and more positive by going to finite covers of $\bar X$ defined using level structures of a local system defined on a Zariski-dense open subset.

  \end{abstract}
\maketitle
\tableofcontents

\section{Introduction}

\subsection{Main results}
Let $X$ be a complex algebraic variety equipped with a variation of complex polarized Hodge structures $\bV = (\cL_{\bC}, \cF^\bullet, h)$. When the associated period map $\tilde X \arrow \cD$ from the universal cover of $X$ to the corresponding period domain has discrete fibers, then $X$ satisfies many hyperbolicity properties. For example, every holomorphic map $\bC  \arrow X$ is constant \cite[Corollary 9.4]{Griffiths-Schmid} and every desingularization of a subvariety of $X$ is of log-general type and has a big logarithmic cotangent bundle \cite{Zuo_neg, Bruni_Crelle, Bruni-Cado}. However, nothing similar can be said in general about the compactifications of $X$, or about the birational properties of $X$ and its subvarieties. This is already seen by looking at the weight $1$ variation of Hodge structures on $\bP^1 \backslash \{0,1, \infty\}$ obtained from the Legendre family of elliptic curves. More generally, if $\mathrm{A}_1(n)$ denotes the coarse moduli space of elliptic curves with a symplectic level-n structure (that is a symplectic basis of the degree $1$ cohomology with coefficients in $\bZ \slash n \bZ$), then its geometric genus\footnote{The geometric genus of an integral curve $C$ is by definition the genus of the smooth projective curve which is birational to $C$.} $g(\mathrm{A}_1(n))$ is bigger than $1$ only when $n > 6$. However, the genus and even the gonality\footnote{The gonality of an integral curve $C$ is the minimal degree of a dominant rational map $C \dasharrow \bP^1$.} of $\mathrm{A}_1(n)$ tends to infinity with $n$, see \cite{Abramovich}. In this paper, we prove a vast generalization of this phenomenon: the results below show that by going to level covers, $X$ and its subvarieties get more and more hyperbolic, and that this holds in a uniform way. We refer to \cite{Nadel, Noguchi91, Hwang-To-uniform_boundedness, Rousseau13, Bakker-Tsimerman, AV16, Bruni_AENS, cadorel2016symmetric, cadorel2018subvarieties, deng2020big} for results of a similar flavour.\\

Assume that $\cL_{\bC} = \cL_\bZ \otimes_{\bZ } \bC$ for a torsion-free $\bZ $-local system $\cL_\bZ $ on $X$ of rank $r = \rk(\cL_{\bC})$. For every positive integer $n$, let $\cL_{\bZ / n \bZ} := \cL \otimes_{\bZ } (\bZ / n \bZ)$ be the induced $\bZ / n \bZ$-local system. We denote by $X(n) = X(\cL_{\bZ / n \bZ})$ the total space of the local system of sets $\cB(n) := \mathcal{I}som_{\bZ / n \bZ}\left((\bZ / n \bZ)_X^{r},\cL_{\bZ / n \bZ} \right)$, so that there is a natural free action of the group $\Gl(r, \bZ / n \bZ)$ on $X(n)$ such that $X$ is identified with the quotient. An $n$-level structure on $\cL_\bZ$ is then by definition a global section of $\cB(n)$, or equivalently a section of the finite \' etale map $X(n) \arrow X$. \\

Our main result in this paper is the following:
\begin{theorem}\label{main result}
Let $\cL_\bZ$ be a torsion-free $\bZ$-local system on a complex algebraic variety $X$, and assume that $\cL_{\bC}$ underlies a variation of complex polarized Hodge structures whose period map has discrete fibres. Then, given a positive number $v$, for all but finitely many prime numbers $p$, every integral subvariety Y of X(p) satisfies $\vol(Y) \geq v$. In particular, for all but finitely many prime numbers $p$, all subvarieties of $X(p)$ are of general type. 
\end{theorem} 
  
Recall that the volume of a line bundle $L$ on a proper variety $X$ of dimension $n$ is the nonnegative real number
\[ \vol(L) := \limsup_{k \rightarrow \infty} \frac{n!}{k^n} \cdot h^0(X, L^{\otimes k})\]
which measures the positivity of $L$ from the point of view of birational geometry. In particular, $\vol(L)$ is positive if and only if the linear system $|kL|$ embbeds $X$ birationally in a projective space for $k$ large enough. The volume $\vol(Y)$ of an integral variety $Y$ is then by definition the volume of the canonical bundle $K_{\bar Y}$ of any smooth proper variety $\bar Y$ birational to $ Y$ (one easily check that this does not depend on $\bar Y$). In particular, $Y$ is of general type if and only if $\vol(Y) >0$. If $Y$ is an integral curve, then $\vol(Y) = 2 \cdot g(Y) - 2$, where $g(Y)$ is the geometric genus of $Y$. Therefore, it follows from Theorem \ref{main result} that the function which associates to a prime number $p$ the minimal genus of an irreducible curve contained in $X(p)$ tends to infinity with $p$. We prove that it is even true for the minimal gonality: 

\begin{cor}\label{increasing gonality}
 With the notations of Theorem \ref{main result}, given an integer $d >0$, for all but finitely many prime numbers $p$, every curve in $X(p)$ has gonality at least $d$.
\end{cor}

\begin{rem}
Consider a finite $\bZ$-module $R \subset \bC$  and a torsion-free $R$-local system $\cL_R$ on a complex algebraic variety $X$, and assume that $\cL_{\bC}$ underlies a variation of complex polarized Hodge structures whose period map has discrete fibres. Then the obvious generalizations of Theorem \ref{main result} and its Corollary \ref{increasing gonality} for finite \'etale covers of $X$ associated to maximal ideals in $R$ are valid and proved exactly in the same way. An easy modification of the proofs permits also to obtain results when $R$ is a $\bZ$-algebra of finite type, but the statements are somewhat cumbersome so we leave them to the interested reader.
\end{rem}

Theorem \ref{main result} and its Corollary \ref{increasing gonality} apply in particular to Shimura varieties, and already in this case they are much stronger than the analogous results in the literature, see section \ref{Applications} for a detailed discussion. For example, if $\cA_g(n)$ denotes the moduli space of $g$-dimensional principally polarized abelian varieties with a symplectic level-n structure, we prove the following

\begin{theorem}[cf. Corollary \ref{Minimal_Gonality_A_g}]
For any positive integer $n$, any curve in $A_g(n)$ has gonality at least $ \lceil \frac{n}{6 g} \rceil $.
\end{theorem}

But our approach permits also to deal with those variations of Hodge structures for which Griffiths horizontality is a non-trivial condition. For example, this permits us to obtain hyperbolicity results for moduli spaces of polarized Calabi-Yau varieties with level structures. (For any prime number $p$, a level-p structure on a smooth projective complex variety $X$ of dimension $d$ is a basis of the $\bF_p$-vector space $\HH^d(X, \bF_p)$.)

\begin{theorem}[cf. Theorem \ref{inf_Torelli}]
For any prime number $p$, let $\cM(p)$ be the moduli stack of polarized Calabi-Yau varieties equipped with a level-p structure (it is a smooth quasi-projective complex variety for $p \geq 3$). Then the minimal volume of an integral variety (resp. the minimal gonality of an integral curve) contained in $\cM(p)$ tends to infinity with $p$.
\end{theorem}

Our proof of Theorem \ref{main result} follows essentially from two independent results that are interesting by themselves:
\begin{enumerate}
\item We prove a higher dimensional version of the now classical Arakelov inequalities: in the situation of the Theorem \ref{main result}, assuming that $X$ is the complementary of a normal crossing divisor $D$ in a smooth proper variety $\bar X$, the ``logarithmic volume'' of all subvarieties of $X$ is uniformely controlled by the volume of the Griffiths parabolic line bundle of $\bV$.

\item We show that any big line bundle on a normal proper algebraic variety $\bar X$ can be made more and more positive by going to finite covers of $\bar X$ defined using level structures of a local system defined on a Zariski-dense open subset.
\end{enumerate}

We now describe these two main ingredients more precisely.

\subsection{Higher dimensional Arakelov inequalities}

Let $\bV = (\cL ,  \cF^{\sbt}, h)$ be a variation of complex polarized Hodge structures on a smooth algebraic variety $X$ with quasi-unipotent monodromies at infinity and denote by $ L_{\bV} := \otimes_{p \in \bZ}  \det \cF^p$ its Griffiths line bundle \cite{GriffithsIII}. Then $L_\bV$ extends canonically as a parabolic line bundle $\bar L_{\bV} $ on any log-smooth compactification of $X$ (when the monodromy at infinity of $\cL$ is unipotent, the parabolic structure on $\bar L_{\bV} $ is trivial and $\bar L_{\bV} $ coincides with the Deligne-Schmid extension of $L_{\bV} $). We refer to section \ref{Preliminaries in Hodge theory} for the definitions.\\

Starting with the results of Arakelov about family of curves parametrized by a curve, there is an abundant literature about the now called Arakelov inequalities for abelian schemes and more generally variations of Hodge structures on curves, cf. \cite{Arakelov, Faltings-Arakelov, Peters-Rigidity, Kim-ABC, Jost-Zuo, peters2000arakelovtype, Viehweg-Arakelov}. When $X$ is the complementary of finitely many points in a smooth projective curve  $\bar X$ and the variation $\bV$ is non-isotrivial (or equivalently $\cL$ has infinite monodromy), their more general versions show that $\deg \bar L_\bV \leq C \cdot (-\chi(X))$ for an explicit positive constant $C$ which depends only on the discrete invariants of $\bV$.  The following result provides a generalized Arakelov inequality which applies to variations of Hodge structures on higher dimensional basis (see also \cite[Theorem 0.2]{Zuo_neg} for another possible generalization).\\

Before giving the statement, recall (cf. Theorem \ref{Griffiths-Schmid metric}) that the Chern curvature form of the Griffiths line bundle $L_{\bV}$ equipped with the hermitian metric induced by $h$ is a closed positive real $(1,1)$-form on $X$. The corresponding K\"ahler pseudo-metric is non-degenerate on the Zariski-open subset of $X$ where the period map is immersive, and has a negative holomorphic sectional curvature $- \gamma_{\bV} <0$.

\begin{theorem}[Higher dimensional Arakelov inequalities]\label{Arakelov inequality} Let $\bar X$ be a smooth proper complex algebraic variety, $D \subset \bar X$ a normal crossing divisor and $\bV = (\cL ,  \cF^{\sbt}, h)$ a variation of complex polarized Hodge structures of length\footnote{\label{footnote}By definition, if $[a,b]$ is the smallest interval such that $\Gr_{\cF}^i \cF^{- \infty} = 0$ for $i \notin [a,b]$, then the length of $\bV$ is the integer $b-a$.} 
 $w$ on $\bar X - D$ with a generically immersive period map. Assume that the local system $\cL$ has quasi-unipotent monodromies around the irreducible components of $D$ and let $\bar L_{\bV}$ be the Griffiths parabolic line bundle of $\bV$. Then
\[ \bar L_{\bV} \leq    \frac{1}{\gamma_{\bV}}  \cdot \K_{\bar X}(D)\footnote{By definition, an inequality $A \leq B$ between two $\bQ$-line bundles $A$ and $B$ on a proper variety means that the $\bQ$-line bundle $B - A$ is pseudoeffective.}.\]
\end{theorem}

This statement is a consequence of a more precise Arakelov inequality that holds at the level of currents, cf. Theorem \ref{current Arakelov inequality}.

\begin{rems}
\begin{enumerate}
\item The holomorphic sectional curvature $\gamma_{\bV}$ depends only on the discrete invariants of $\bV$ (or equivalently on the corresponding period domain), and it satisfies (cf. Theorem \ref{Griffiths-Schmid metric})
\[ \frac{1}{\gamma_{\bV}}   \leqslant \frac{w^2}{4} \cdot \rank( \cL),\]
so that we get an effective Arakelov inquality. Note that if $(\cL, \cF^{\sbt}, h)$ is isomorphic to its dual (this holds for example if it is a variation of real polarized Hodge structures), then
\[ \frac{w^2}{4}  \cdot \rank( \cL) =   \frac{w}{2} \cdot \sum_{i= 1}^p \rank( \cF^p) . \]
\item Since in that situation the cohomology class of the Griffiths parabolic line bundle is big (cf. Proposition \ref{Griffiths_line_bundle_big}), we recover that $X$ is of log-general type \cite{Zuo_neg, Bruni_Crelle, Bruni-Cado}.
\end{enumerate}
\end{rems}

\begin{ex} 
Let $X$ be a smooth complex variety and $\cA \arrow X$ an abelian scheme of relative dimension $g$ and maximal variation. Let $\bar X$ be a smooth compactification of $X$ such that $D := \bar X - X$ is a normal crossing divisor, and assume that $\cA$ extends as a semi-abelian scheme $\bar \cA \arrow \bar X$ with zero section $e : \bar X \arrow \bar \cA$. Applying Theorem \ref{Arakelov inequality} to the variation of Hodge structures coming from the relative cohomology in degree $1$, we get the inequality 
\[ e^\ast \Omega^g_{\bar \cA / \bar X} \leq    \frac{g}{2}  \cdot \K_{\bar X}(D).\]
This generalizes the well-known inequality over curves \cite{Faltings-Arakelov, Kim-ABC}.
\end{ex}

Since pseudoeffective divisors are dual to movable curves \cite{BDPP}, Theorem \ref{Arakelov inequality} is equivalent to the following statement.
\begin{theorem} With the notations of Theorem \ref{Arakelov inequality}, the following inequality holds for every movable curve $C \subset \bar X$ with class $[C] \in N_1(\bar X)$:

\[    \bar L_{\bV} \cdot [C] \leq    \frac{1}{\gamma_{\bV}}   \cdot  \K_{\bar X}(D) \cdot [C]. \]

In particular, this applies with $[C] = [H]^{\dim X -1}$ for $H$ an ample divisor.
\end{theorem}

We now derive from Theorem \ref{Arakelov inequality} two immediate corollaries.

\begin{cor} With the notations of Theorem \ref{Arakelov inequality}, the following inequality holds:

\[    \vol (\bar L_{\bV}) \leq    \frac{1}{(\gamma_{\bV})^{\dim X}} \cdot \vol( \K_{\bar X}(D)). \]
\end{cor}

\begin{cor}\label{criterion_general_type}
With the notations of Theorem \ref{Arakelov inequality}, if we assume moreover that the $\bQ$-line bundle $\bar L_{\bV}(-   \frac{1}{\gamma_{\bV}}   \cdot D)$ is big, then $\bar X$ is of general type. More generally, every integral subvariety of $X$ not contained in the augmented base locus of $\bar L_{\bV}(-  \frac{1}{\gamma_{\bV}}   \cdot D)$ is of general type.
\end{cor}

\begin{rems}
\begin{enumerate}
\item Most of these results were already proved by differents methods when $X$ is the quotient of a bounded symmetric domain by a torsion-free arithmetic lattice, cf. \cite{Bruni_AENS, cadorel2018subvarieties}.
\item Assuming that $\bV$ is a variation of \textit{real} polarized Hodge structures, and observing that $\bar L_{\bV}$ is the smallest nonzero Hodge subbundle of the variation of real polarized Hodge structures $ \otimes_{p \in \bZ}  \Lambda^{\rk \cF^p} \bV$, it follows from \cite[Theorem 3.3]{Bruni_AENS} that the cotangent bundle of $\bar X$ is even big under the stronger assumption that the $\bQ$-line bundle $\bar L_{\bV}(- \frac{w^2 \cdot \rk \cL}{2} \cdot D)$ is big.
\end{enumerate}
\end{rems}

\subsection{Increasing positivity through level structures}

We consider a normal projective variety $\bar X$ equipped with a big line bundle $L$, and a torsion-free $\bZ$-local system $\cL_\bZ$ defined on a Zariski-dense open subset $X \subseteq \bar X$. We assume that the local system $\cL_\bZ$ is large, i.e. its pullback by any non-constant algebraic map is non-trivial \cite{Kollar-shafa}. This is true for example if $\cL_{\bC} :=  \cL_\bZ \otimes_\bZ \bC$ underlies a variation of complex polarized Hodge structures whose period map has discrete fibres, cf. proposition \ref{VPHS_verify_the_assumptions}.\\

As before, for every positive integer $n$, we denote by $X(n)$ the finite \'{e}tale cover of $X$ that trivializes the local system $\cL_\bZ \otimes_\bZ (\bZ / n \bZ)$. Let also $\bar X(n)$ be the normalization of $\bar X$ in the total ring of fractions of $X(n)$ and $\pi_n : \bar X(n) \arrow \bar X$ the projection.\\

Informally speaking, the results of this section show that the pullback of $L$ to $ \bar X(n) $ gets more and more positive as $n$ gets big. To state them, let us recall the definition of some invariants attached to $L$ commonly used to measure its positivity.\\

First introduced in \cite{Nakamaye}, the augmented base locus (or non-ample locus) of $L$ is defined as 
\[ \mathbf{B}_+(L) = \cap \, \Supp (D)\]
over all decompositions $L ^{\otimes n} = A(D)$ with $A$ an ample line bundle, $D \subset \bar X$ an effective divisor and $n$ a positive integer. Following \cite{ELMNP09}, we define for every subvariety $\bar Z \subseteq \bar X$ the restricted volume of $L$ on $\bar Z$ as
\[ \vol_{\bar X | \bar Z}(L) = \limsup_{k \rightarrow \infty} \frac{d!}{k^d} \cdot \dim \HH^0(\bar X| \bar Z, L^{\otimes k}),\]

where $d := \dim \bar Z$ and $ \HH^0(\bar X| \bar Z, L^{\otimes k})$ denotes the image of the restriction map $ \HH^0(\bar X, L^{\otimes k}) \arrow  \HH^0( \bar Z,( L_{|\bar Z})^{\otimes k})$. When $L$ is nef, we have the equalities $\vol_{\bar X | \bar Z}(L) = \vol(L_{|\bar Z}) = L^{\dim \bar Z} \cdot \bar Z $.  In general, thanks to \cite{Nakamaye}, \cite[Thm C]{ELMNP09} and \cite[Thm B]{BCL}, the augmented base locus of $L$ coincides with the union of the irreducible subvarieties $\bar Z \subseteq \bar X$ such that $\vol_{\bar X | \bar Z}(L) = 0$.

\begin{theorem}\label{increasing positivity}
Let $\bar X$ be a normal projective variety equipped with a big line bundle $L$, and let $\cL_\bZ$ be a torsion-free $\bZ$-local system defined on a Zariski-dense open subset $X \subseteq \bar X$. Assume that the local system $\cL_\bZ$ is large. Then for any $v >0$, for all but finitely many positive integers $n$, every subvariety $\bar Z$ of $\bar X(n)$ intersecting $X(n)$ satisfies $\vol_{\bar X(n) | \bar Z}(\pi_n ^\ast L) \geq v $, unless its projection is contained in $\mathbf{B}_+(L) $.  
\end{theorem}

Since the obvious inequality $ \vol((\pi_n ^\ast L)_ {|\bar Z}) \geq \vol_{\bar X(n) | \bar Z}(\pi_n ^\ast L)$ holds for every subvariety $\bar Z$ of $\bar X(n)$, we get immediately the following result.

\begin{cor}\label{increasing positivity-bis}
Let $\bar X$, $L$ and $\cL_\bZ$ as above. For any $v >0$, for all but finitely many positive integers $n$, every subvariety $\bar Z$ of $\bar X(n)$ intersecting $X(n)$ satisfies $ \vol((\pi_n ^\ast L)_ {|\bar Z})  \geq v $, unless its projection is contained in $\mathbf{B}_+(L)$.  
\end{cor}

\begin{rem}
The proof of Theorem \ref{increasing positivity} can be easily adapted to other similar contexts. One can for example consider a normal projective variety $\bar X$ defined over a field of positive characteristic and replace $\cL_\bZ$ by a $\hat{\bZ}$-representation of the \'etale fundamental group of a Zariski-dense open subset $X \subseteq \bar X$. Or one can also consider positive forms on normal compact K\"ahler analytic spaces.
\end{rem}

\subsection{Acknowledgments} 
I am grateful to Ben Bakker for several helpful discussions, and especially one that leaded to Corollary \ref{increasing gonality}. I would like also to thank Damian Brotbek for our collaboration \cite{Brotbek-Brunebarbe} that prompted many improvements in this paper.

\section{Preliminaries in Hodge theory}\label{Preliminaries in Hodge theory}

\subsection{Variations of Hodge structures}
A complex polarized Hodge structure (of weight zero) on a finite-dimensional complex vector space $V$ is the data of a non-degenerate hermitian form $h$ on $V$ and of a decomposition $V = \bigoplus_{p \in \bZ} {V}^p$ which is orthogonal for $h$ and such that the restriction of $h$ to $ {V}^p$ is positive definite for $p$ even and negative definite for $p$ odd. The associated Hodge metric on $V$ is the positive-definite hermitian metric $h_H$ obtained from $h$ by imposing that the Hodge decomposition $V = \bigoplus_{p \in \bZ} {V}^p$ is $h_H$-orthogonal and setting $h_H := (-1)^p \cdot h$ on $V^p$. The associated Hodge filtration is the decreasing finite filtration $\{F^{\sbt}\}$ on $V$ defined by $F^p := \bigoplus_{q \geq p} V^q$. Note that the Hodge decomposition is determined by the Hodge filtration thanks to the formula $V^p = F^p \cap (F^{p+1})^\perp$. Here $(F^{p+1})^\perp$ denotes the orthogonal with respect to the polarization $h$, which is clearly equal to the orthogonal of $F^{p+1}$ with respect to the Hodge metric $h_H$.\\

A variation of complex polarized Hodge structures ($\bC$-VPHS) on a complex analytic space $X$ is composed by a complex local system $\cL$ on $X$ equipped with a non-degenerate hermitian form $h : \cL \otimes_{\bC} \bar \cL \arrow \bC_X$ and a locally split finite filtration $\cF^\bullet$ of $\cL \otimes_{\bC}\cO_X$ by analytic coherent subsheaves which satisfies Griffiths transversality on the reduced regular locus of $X$ and which is fiberwise a polarized complex pure Hodge structure. Observe that the Hodge metrics on the fibres gather as a smooth (but not flat) positive definite hermitian metric on the vector bundle $\cL \otimes_{\bC}\cO_X$.\\

As usual, a $\bC$-VPHS on a complex analytic space $X$ defines a period map from the universal cover of $X$ to a classifying space for polarized Hodge structures. We refer to \cite{GriffithsIII, Schmid73} for more details.


\subsection{Meromorphic extension of a variation of Hodge structures}

Let $\bar X$ be a complex manifold, $D \subset \bar X$ be a normal crossing divisor and $\bV = (\cL ,  \cF^{\sbt}, h)$ be a $\bC$-VPHS on $X$ with quasi-unipotent monodromies around the irreducible components $\{D_i\}_{i \in I}$ of $D$. Let $\cV := \cF^{- \infty}$ be the underlying locally-free $\cO_X$-module and $\nabla = id \otimes d$ the flat connection coming from the identification $\cV = \cL \otimes_{\bC} \cO_X$.\\

By the work of Deligne \cite{Deligne_book}, the holomorphic vector bundle $\cV$ extends uniquely (up to a unique isomorphism) as a meromorphic bundle $\bar \cV$ on $(\bar X, D)$ in which the connection $\nabla$ is regular meromorphic. For every $\alpha \in \bR^I$, let $\cV^\alpha$ be the unique locally-free $\cO_{\bar X}$-module of finite rank contained in $\bar \cV$ such that $\nabla$ induces a connection with logarithmic singularities $\nabla : \cV^\alpha \arrow \cV^\alpha \otimes \Omega^1_{\bar X}( \log D)$ such that the real part of the eigenvalues of the residue of $\nabla$ along $D_i$ belongs to $[\alpha_i,\alpha_i + 1)$, cf. \cite[Proposition 5.4]{Deligne_book}. We obtain in this way a filtered regular meromorphic connection bundle that we call the Deligne's extension of $(\cV, \nabla)$.\\

For every integer $p$ and every $\alpha$, the sheaf $\cF^p$ extends as a locally split subsheaf of $\cV^\alpha$ by setting: $\cF^p \cV^\alpha := \cV^\alpha \cap j_\ast \cF^p$.
This defines a filtration $\{ \cF^p \, \bar \cV \}_{p \in \bZ}$ of $\bar \cV$ in the category of parabolic vector bundles on $(\bar X, D)$. The data of the Deligne's extension of $(\cV, \nabla)$ together with its extended Hodge filtration form a meromorphic variation of Hodge structures in the sense of \cite{Bruni_semipositivity}, that we call the meromorphic extension of $\bV$.\\

If $\bar Y$ is another complex manifold, $E \subset \bar Y$ a normal crossing divisor and $f : \bar Y \arrow \bar X$ a holomorphic map such that $f^{-1}(D) \subseteq E$, then the meromoprhic extension of the pullback variation $f^\ast \bV$ is equal to the pullback of the meromorphic extension of $\bV$, see \cite{Bruni_semipositivity}.

\subsection{Computing the Chern classes using the Hodge metric}

The next result is a slight generalization of results of Cattani-Kaplan-Schmid \cite[Corollary 5.23]{CKS} and Koll\'ar \cite[Theorem 5.1]{Kollar-subadditivity}.
\begin{theorem}\label{Chern class and current}
Let $\bar X$ be a smooth proper complex algebraic variety, $D \subset \bar X$ a normal crossing divisor and $\bV = (\cL ,  \cF^{\sbt}, h)$ a $\bC$-VPHS on $X$ with quasi-unipotent monodromies around the irreducible components of $D$. Let $L$ be a locally split parabolic line subbundle of the meromorphic extension of $\bV$ on $\bar X$ and equipp $L_{|X}$ with the smooth hermitian metric $h_H$ induced by the Hodge metric of $\bV$. Then the first Chern form
\[ C_1(L_{|X}, h_H) := \frac{i}{2\pi} \cdot \Theta_{h_H}( L_{|X}) \]
extends as a closed current on $\bar X$ whose class in $\HH^2(\bar X, \bR)$ is equal to the first parabolic Chern class of $L$.
\end{theorem} 

\begin{rem}
Following the same argumentation as in \cite{Kollar-subadditivity}, one easily sees that this implies more generally that any homogeneous polynomial in the Chern forms of some locally abelian parabolic bundles that are locally split subquotients of the meromorphic extension of $\bV$ extends as a closed current on $\bar X$ whose class in $\HH^2(\bar X, \bR)$ is equal to the polynomial in the corresponding parabolic Chern classes.
 \end{rem}

\begin{proof}
When the monodromies around the irreducible components of $D$ are unipotent, this is a result of Koll\'ar \cite[Theorem 5.1]{Kollar-subadditivity}. More precisely, he proves the following. Let $h^\prime$ be any smooth hermitian metric on $L$ with Chern connection $\theta_{h^\prime}$ and curvature $\Theta_{h^\prime}$. Let also $\theta_{h_H}$ and $\Theta_{h_H}$ be the Chern connection and curvature of the line bundle $L_{|X}$ equipped with the Hodge metric. Then $\theta_{h_H} - \theta_{h^\prime}$ is a well-defined $1$-form on $X$, and the following equality of $2$-forms $\Theta_{h_H} = \Theta_{h^\prime} + d( \theta_{h_H} - \theta_{h^\prime})$ holds on $X$. Now, as a consequence of \cite[Proposition 5.15]{Kollar-subadditivity} and \cite[Corollary 5.17]{Kollar-subadditivity}, it follows that $ \Theta_{h_H}$ and $d( \theta_{h_H} - \theta_{h^\prime})$ both extend as closed $(1,1)$-currents  $[ \Theta_{h_H}]$ and $[d( \theta_{h_H} - \theta_{h^\prime})]$ on $\bar X$, that $\theta_{h_H} - \theta_{h^\prime}$ extends as a $(0,1)$-current $[\theta_{h_H} - \theta_{h^\prime}]$ on $\bar X$, and that $d[\theta_{h_H} - \theta_{h^\prime}] = [d( \theta_{h_H} - \theta_{h^\prime})]$. Therefore the two closed currents $\Theta_{h^\prime}$ and $\Theta_{h_H}$ on $\bar X$ defines the same cohomology class and the result follows in this case.\\

In general, the monodromy group of $\cL$ being finitely generated, it admits a net finite index subgroup by \cite[Corollaire 17.7]{Borel-groupes_arithmetiques}, that we assume to be normal without loss of generality. Let $f : Y \arrow X$ be the associated finite \'etale Galois cover. Thanks to Riemann existence theorem and Hironaka desingularization theorem, there exist a smooth proper algebraic variety $\bar Y$, a normal crossing divisor $ E \subset \bar Y$ and a morphism $\bar f: (\bar Y, E) \arrow (\bar X,D)$ that extends $f : Y = \bar Y - E \arrow X$. By construction, the monodromies of the local system $f^\ast \cL$ around the irreducible components of $E$ are unipotent.\\
 
Let $\alpha $ be a closed smooth $(1,1)-$form on $\bar X$ whose class in $\HH^2(\bar X, \bR)$ is equal to the first parabolic Chern class of $L$. The closed smooth form $f^\ast \alpha$ represents the class $f^\ast c_{1}(L) = c_{1}(f^\ast L) = c_{1}(^\diamond (f^\ast L)) $. Therefore, $f^\ast \alpha$ is the Chern curvature of a smooth hermitian metric on $^\diamond (f^\ast L)$. Denoting by $\Theta_{h_H}$ the Chern curvature of the line bundle $L_{|X}$ equipped with the Hodge metric, we know from the preceding discussion that
\begin{enumerate}
\item $f^\ast \Theta_{h_H}$ extends as a closed current $[f^\ast \Theta_{h_H}]$ on $\bar Y$, 
\item  there exists a smooth form $\beta$ on $Y$ that extends as a current $[\beta]$ on $\bar Y$ and such that $[f^\ast \alpha] = [ f^\ast \Theta_{h_H} ] + d [\beta]$ and $d [\beta] = [d \beta]$.
\end{enumerate}
By replacing $\beta$ with its average over the fibers of the Galois cover $f: Y \arrow X$, we can assume that $\beta = f^\ast \gamma$ for a smooth form $\gamma$ on $X$, so that $\alpha = \Theta_{h_H} + d \gamma$ on $X$.  Since $\beta$ is locally integrable, it follows that $\gamma$ is locally integrable, hence it extends as a current $[\gamma]$ on $\bar X$. Similarly, $  \Theta_{h_H}$ extends as a current $ [ \Theta_{h_H} ]$ on $\bar X$. The equality $d \gamma = \alpha - \Theta_{h_H} $ shows then that $d \gamma$ extends as well as a current $[d \gamma]$ on $\bar X$. Since the current $[f^\ast \Theta_{h_H}]$ is closed, and since for any smooth form $\delta$ of degree $2 \cdot \dim_{\bC} \bar X - 2$ on $\bar X$ one has
\begin{equation*}
[f^\ast \Theta_{h_H}]( f^\ast \delta)  = \int_Y  f^\ast \Theta_{h_H} \wedge f^\ast \delta   = \deg(f) \cdot \int_X  \Theta_{h_H} \wedge \delta =  \deg(f) \cdot [ \Theta_{h_H}](\delta) ,
\end{equation*}
it follows that the current $[ \Theta_{h_H}]$ is closed too. Similarly, one proves that $d [\gamma] = [d \gamma]$ using that $d [\beta] = [d \beta]$ and $\beta = f^\ast \gamma$. Finally, we get that $[\alpha]  = [ \Theta_{h_H} ] + [d\gamma]  = [ \Theta_{h_H} ] + d [\gamma]$, so that $ [ \Theta_{h_H} ] $ and $\alpha$ defines the same cohomology class $c_1(L)$.
\end{proof}

\subsection{The Griffiths line bundle of a variation of Hodge structures}
Following \cite{GriffithsIII}, if $\bV = (\cL, \cF^{\sbt}, h)$ is a $\bC$-VPHS on a complex analytic space $X$, we define the Griffiths line bundle of $\bV$ by the formula $ L_{\bV} := \otimes_{p \in \bZ}  \det \cF^p$.

\begin{theorem}[{cf. \cite[Theorem 1.9]{Brotbek-Brunebarbe}}]\label{Griffiths-Schmid metric}
Let $\bV = (\cL, \cF^{\sbt}, h)$ be a variation of complex polarized Hodge structures of length $w$ on a complex manifold $S$. The Chern curvature form of its Griffiths line bundle $L_{\bV}$ equipped with the hermitian metric induced by $h$ is a closed positive real $(1,1)$-form on $S$. The corresponding K\"ahler pseudo-metric is non-degenerate on the Zariski-open subset of $S$ where the period map is immersive, has non-positive holomorphic bisectional curvature and its holomorphic sectional curvature $- \gamma_{\bV}$ satisfies 
\[ \frac{1}{\gamma_{\bV}}   \leqslant \frac{w^2}{4} \cdot \rank( \cL). \] 
\end{theorem}

Note that when $\bV$ is isomorphic to its dual (this holds for example if $\bV$ is a variation of \textit{real} polarized Hodge structures), then 
\[ \frac{w^2}{4}  \cdot \rank( \cL) =   \frac{w}{2} \cdot \sum_{i= 1}^p \rank( \cF^p) . \]

Let $X$ be the complementary of a normal crossing divisor $D$ in a smooth proper complex algebraic variety $\bar X$ and $\bV = (\cL ,  \cF^{\sbt}, h)$ a $\bC$-VPHS on $X$ with quasi-unipotent monodromies around the irreducible components of $D$. If $(\bar \cV, \nabla, \cF^\bullet \bar \cV, h)$ is the meromorphic extension of $\bV$, then we define the Griffiths parabolic line bundle of $\bV$ as
\[ \bar L_{\bV} := \otimes_{p \in \bZ}  \det \cF^p \, \bar \cV .\]
By definition it is a parabolic line bundle on $(\bar X, D)$ which coincides on $X$ with the Griffiths line bundle of $\bV$. If $\bar Y$ is another complex manifold, $E \subset \bar Y$ a normal crossing divisor and $f : \bar Y \arrow \bar X$ a holomorphic map such that $f^{-1}(D) \subseteq E$, then the Griffiths line bundle of the pullback variation $f^\ast \bV$ is equal to $f^\ast L_{\bV}$. The following result is a direct consequence of Theorem \ref{Chern class and current} and Theorem \ref{Griffiths-Schmid metric} that we state for reference:

\begin{prop}\label{Griffiths_line_bundle_big}
Let $X$ be the complementary of a normal crossing divisor $D$ in a smooth proper complex algebraic variety $\bar X$ and $\bV = (\cL ,  \cF^{\sbt}, h)$ a $\bC$-VPHS on $X$ with quasi-unipotent monodromies around the irreducible components of $D$. If the period map of $\bV$ is generically immersive, then the cohomology class of the Griffiths parabolic line bundle of $\bV$ is big.
If moreover the period map of $\bV$ has discrete fibres, then the cohomology class of the Griffiths parabolic line bundle of $\bV$ is ample modulo $D$,  i.e. $\mathbf{B}_+(\bar L_{\bV}) \subset D$.
\end{prop}

See also \cite{Bruni_semipositivity} where it is proved that the Griffiths parabolic line bundle is always nef.


\section{Higher dimensional Arakelov inequalities}

The goal of this section is to prove an Arakelov inequality at the level of currents that will imply Theorem \ref{Arakelov inequality}.
\subsection{An Arakelov inequality between currents}
Let $\bar X$ be a complex manifold of dimension $n$, $D \subset \bar X$ a normal crossing divisor and $\bV = (\cL ,  \cF^{\sbt}, h)$ a $\bC$-VPHS of length $w$ on $X = \bar X - D$. Thanks to Theorem \ref{Griffiths-Schmid metric}, the Chern curvature form of the Griffiths line bundle $L_{\bV}$ of $\bV$ equipped with its Hodge metric $h_H$ defines a  K\"ahler pseudo-metric on $X$. Assume that the Zariski-open subset $\mathring{X}$ of $X$ where the period map is immersive is non-empty and equipp $\mathring{X}$ with the K\"ahler metric whose K\"ahler form is the Chern curvature form of the Griffiths line bundle $L_{\bV}$ of $\bV$ equipped with its Hodge metric. We denote by $- \gamma_{\bV}$ its holomorphic sectional curvature, so that $\frac{1}{\gamma_{\bV}}   \leqslant \frac{w^2}{4} \cdot \rank( \cL)$. Finally, we denote by $h_{GS}$ the induced pseudo-metric on $X$ and $\K_X$.

\begin{theorem}\label{current Arakelov inequality}
Notations as above. Then the following inequalities 
\[ C_1(L_{\bV}, h_H)  \leq \frac{1}{\gamma_{\bV}} \cdot C_1(\K_{X}, h_{GS}) \leq \frac{w^2 \cdot \rk{\cL}}{4} \cdot C_1(\K_{X}, h_{GS})\]
between positive $(1,1)$-forms hold pointwise on $\mathring{X}$. Moreover, the first Chern form $C_1(L_{\bV}, h_H) $ extends as a closed positive current $[C_1(L_{\bV}, h) ]$ on $\bar X$, the pseudo-metric $h_{GS}$ on $\K_X$ extends a singular metric on $\K_{\bar X}(D)$ and the inequalities
\[ [C_1(L_{\bV}, h_H) ]  \leq  \frac{1}{\gamma_{\bV}} \cdot C_1(\K_{\bar X}(D), h_{GS}) \leq  \frac{w^2 \cdot \rk{\cL}}{4} \cdot C_1(\K_{\bar X}(D), h_{GS})\]
hold in the sense of currents on $\bar X$.
\end{theorem}

\subsection{Preliminaries on K\"ahler metrics}
Consider a K\"ahler manifold $(M,g)$ with K\"ahler form $\omega$. We denote by $J$ the real operator on the tangent vector bundle that defines the complex structure. We denote also by $\nabla$ the Levi-Civita connection of the underlying Riemannian manifold and $R \in \cA^2(\End(TM))$ its curvature. Recall that the Ricci curvature tensor $r$ is defined by 
\[ r(X,Y) = \mathrm{tr}(W \mapsto R(W,X)\cdot Y) \]
for any real tangent vectors $X, Y,W$. The Ricci tensor of a K\"ahler metric is a real symmetric bilinear form of type $(1,1)$. We denote by $\Ric_\omega$ the associated $2$-form defined by $\Ric_\omega(u, v) := r(J(u) , v)$. It is a closed real $(1,1)$-form which is called the Ricci form. Recall the equality 
\[       \Ric_\omega =          i \cdot \Theta (\K_M) , \] 
where $\Theta(\K_M)$ denote the Chern curvature of the canonical bundle $\K_M =\Omega^n_M$ of $M$ equipped with the smooth hermitian metric induced by the K\"ahler metric on $M$.\\

Recall that the holomorphic bisectional  curvature at $x \in M$ is  defined as 
\[ H(X,Y) := g(R(X,JX)\cdot JY, Y)\]
for every unit vector $X$ and $Y$ in $ T_x M$, and that the holomorphic sectional  curvature at $x \in M$ is  defined as $H(X) := H(X,X)$ for every unit vector $X$ in $ T_x M$.

\begin{prop}\label{Ricci curvature and sectional curvature}
Let $M$ be a K\"ahler manifold with K\"ahler form $\omega$. Assume that its holomorphic bisectional curvature is non-positive and that its holomorphic sectional curvature is bounded from above by $- \gamma < 0$. Then $\Ric_\omega \leq - \gamma \cdot \omega$.
\end{prop}
\begin{proof}
Let $x \in M$ and $X \in T_x M$ a unit vector. Take an orthonormal basis $(e_1, \cdots , e_n)$ of $T_x M$ such that $e_1 = X$. Then
\[ r(X,X) = \sum \limits_{j=1}^n H(e_j, X)  \leq H(e_1, e_1) \leq - \gamma .\]
\end{proof}
\subsection{Criterions for extending a metric to the log-canonical bundle}
\begin{prop}\label{Criterions for extending a metric to the log-canonical bundle}
Let $X$ be the complementary of a normal crossing divisor $D$ in a complex manifold $\bar X$ of dimension $n$. Let $\omega$ be a closed smooth real $(1,1)$-form on $X$ which is positive on a dense Zariski-open subset $\mathring{X}$. Assume that there exists a constant $\gamma >0$ such that $\Ric_\omega \leq - \gamma \cdot \omega$ pointwise on $\mathring{X}$. Then 
\begin{enumerate}
\item the smooth hermitian metric on $\K_{\mathring{X}}$ induced by $\omega_{|\mathring{X}}$ extends as a singular metric $\tilde{h}$ with positive curvature on $\K_X(D)$,
\item the smooth differential form $\omega$ extends as a positive real $(1,1)$-current $\tilde{\omega}$ on $\bar X$,
\item and one has the following inequality between currents on $\bar X$:
\[  \gamma \cdot  \tilde{\omega} \leq  \Theta_{\tilde{h}}(\K_X(D))  . \]
\end{enumerate}
\end{prop}

Observe that the second and third points are direct consequences of the first. The first point will be a consequence of a more general extension result for singular metric on the canonical bundle that we now explain.\\

If $X$ is a complex manifold of dimension $n$, then the smooth hermitian metrics on its canonical bundle $\K_X$ correspond bijectively to the volume forms as follows: fixing a trivializing holomorphic section $\eta$ of $ \K_X$, we associate to any metric $h$ on $\K_X$ the volume form $\frac{1}{| \eta|^2_h} \cdot \frac{i ^{n ^2}}{2^n} \cdot \eta \wedge \bar \eta$. Through this correspondance, the metric on $\K_X$ induced by a K\"ahler metric on $X$ with K\"ahler form $\omega$ is sent to the volume form $\frac{1}{n!} \cdot \omega^n$.\\

We will consider more generally some pseudo-volume forms on $X$. They are by definition continous tensors of type $(n, n)$ which are $\cC^\infty$-volume forms on a dense Zariski-open subset. If $\Psi$ is a pseudo-volume form, we denote by $\Zero(\Psi)$ the set where it vanishes (which by definition is contained in a thin closed analytic subset).

\begin{prop}
Let $X$ be the complementary of a normal crossing divisor $D$ in a complex manifold $\bar X$ of dimension $n$. Consider a pseudo-volume form $\Psi$ on $X$ and let $h$ be the associated smooth metric on the canonical bundle on $X - \Zero(\Psi)$. Denoting by $\Theta_h$ its curvature, assume that there exists a constant $A >0$ such that the inequalities $\Theta_h \geq 0$ and $ \Theta_h^n  \geq A \cdot \Psi$ hold pointwise on $X - \Zero(\Psi)$. Then $h$ extends as a singular metric with positive curvature on $\K_X(D)$.
\end{prop}

\begin{proof}
It is sufficient to consider the local situation where $\bar X = \bD^n$ and $X = (\bD^\ast)^r \times \bD^{n-r}$ for some $0 < r <n$. We denote by $(z_1, \cdots, z_n)$ the canonical coordinates on $\bD^n$ and by $| \cdot |_P$ the metric on $\K_X$ induced by the Poincar\' e metric on $X$. Then by applying the lemma below to the universal cover of $X$, we get that

\[ \frac{A}{| dz_1 \wedge \cdots \wedge dz_n |^2_h} \leq  \frac{1}{n!} \cdot  \frac{1}{| dz_1 \wedge \cdots \wedge dz_n |^2_P}. \]

It follows that the function 
\[ (z_1, \cdots, z_n) \mapsto - \log | \frac{1}{z_1 \cdots z_r} \cdot  (dz_1 \wedge \cdots \wedge dz_n) |^2_h \]
goes to $- \infty$ at $0$, so that it is a fortiori bounded from above in a neighborhood of $0$ in $\bar X$. Since by assumption it is pluri-subharmonic on the dense Zariski-open subset $X - \Zero(\Psi)$, it extends as a pluri-subharmonic function on $\bar X$. This exactly means that $h$ extends as a singular metric with positive curvature on $\K_X(D)$.
\end{proof}
\begin{lem}[{\cite[Lemma 2.2.8]{Noguchi-Ochiai}}]
Consider a pseudo-volume form $\Psi$ on the polydisk $\bD^n$. Assume that there exists a constant $A >0$ such that the following inequalities hold pointwise outside $\Zero(\Psi)$:
\begin{enumerate}
\item $\Theta_{\Psi} \geq 0$,
\item $    \Theta_{\Psi}^n  \geq A \cdot \Psi$ for a constant $A >0$.
\end{enumerate}
Then $A  \cdot \Psi \leq \frac{1}{n!} \cdot \vol_P$.
\end{lem}

\subsection{Proof of Theorem \ref{current Arakelov inequality} and Theorem \ref{Arakelov inequality}}

Let us first prove Theorem \ref{current Arakelov inequality}. We keep the notations of the statement. If we denote by $\omega = C_1(L_{\bV}, h_H) $ the Chern curvature of the Griffiths line bundle $L_{\bV}$ of $\bV$ equipped with its Hodge metric, then $\omega$ is a closed smooth real $(1,1)$-form on $X$ which is positive on the dense Zariski-open subset $\mathring{X}$ of $X$  where the period map is immersive. Moreover, the inequality $\Ric_\omega \leq - \gamma_{\bV} \cdot \omega$ holds pointwise on $\mathring{X}$ with $\frac{1}{\gamma_{\bV}}  \leq \frac{w^2}{4} \cdot \rank( \cL)$ thanks to Theorem \ref{Griffiths-Schmid metric} and Proposition \ref{Ricci curvature and sectional curvature}. Therefore Theorem \ref{current Arakelov inequality} follows by applying Proposition \ref{Criterions for extending a metric to the log-canonical bundle}.\\

Assuming moreover that $\bar X$ is a smooth proper algebraic variety and that $\cL$ has quasi-unipotent monodromies around the irreducible components of $D$, Theorem \ref{Arakelov inequality} follows from Theorem \ref{current Arakelov inequality} since the closed current $[ C_1(L_{\bV}, h_H) ]$ represents the first Chern class of the parabolic line bundle $\bar L_{\bV}$ thanks to Theorem \ref{Chern class and current}.

\section{Level structures covers}

The goal of this section is to prove Theorem \ref{increasing positivity} and its corollaries. This will be done in several steps.

\subsection{Preliminaries on families of normal cycles}

Following Koll\'ar \cite[definition 2.1]{Kollar-shafa} a normal cycle on a normal algebraic variety $X$ is by definition an irreducible and normal algebraic variety $W$ together with a finite
morphism $w : W \arrow X$ which is birational to its image. The normalization of any closed irreducible subvariety of $X$ yields a normal cycle on $X$, and all normal cycles on $X$ are obtained in this way.\\

A family of normal cycles on a normal algebraic variety $X$ \cite[definition 2.2]{Kollar-shafa} is a diagram 
 \[\xymatrix{
U \ar[r]^u  \ar[d]^p& X \\
S  & 
}\]
where
\begin{enumerate}
\item every connected component of $U$ and of $S$ is of finite type and
reduced (but there can be countably many such components);
\item $p$ is flat with irreducible, geometrically reduced and normal
fibers;
\item for every $s\in S$, $u_{| p^{-1}(S)} :  p^{-1}(S) \arrow X$ is a normal cycle.
\end{enumerate}

\begin{prop}\label{universal family}
Let $\bar X$ be a normal complex projective variety equipped with an ample line bundle $L$. Let $X \subset \bar X$ be a Zariski-dense open subset and $v >0$ a real number. Then there exists a family $(p:U \arrow S, u : U \arrow \bar X)$ of normal cycles on $\bar X$ parametrized by a smooth algebraic variety $S$ with finitely many irreducible components such that
\begin{enumerate}
\item every normal cycle $\bar W \arrow \bar X$ such that $ L^{\dim \bar W} \cdot \bar W \leq v $ appears as a fiber of $p$, and
\item the map $p_{| u^{-1}(X)} : u^{-1}(X) \arrow S$ is a topological fiber bundle.
\end{enumerate}
\end{prop}

\begin{proof}
Same proof as \cite[Proposition 2.8]{Kollar-shafa}, but replacing $\mathrm{Chow}(\bar X)$ with the union of those (finitely many) irreducible components consisting in cycles $W$ in $\bar X$ that satisfy $ L^{\dim \bar W} \cdot \bar W \leq v $.
\end{proof}

\begin{cor}\label{finiteness monodromy representations}
Let $\bar X$ be a normal complex projective variety equipped with an ample line bundle $L$. Let $\cL_\bZ$ be a torsion-free $\bZ$-local system of rank $r$ defined on a Zariski-dense open subset $X \subset \bar X$.
For any real number $v > 0$, there are, up to conjugation by an element of $ \Gl(r, \bZ)$, only finitely many subgroups of $ \Gl(r, \bZ)$ obtained as the monodromy group of the local system induced by $\cL$ on $w^{-1}(X)$ for a normal cycle $w : \bar W \arrow \bar X$ such that $ L^{\dim \bar W} \cdot \bar W \leq v $.
\end{cor}


\subsection{Setting}
We fix an irreducible normal complex projective variety $\bar X$, a big line bundle $L$ on $\bar X$ and a torsion-free $\bZ$-local system $\cL_\bZ$ on a Zariski-dense open subset $X \subseteq \bar X$. Throughout this section we assume that the local system $\cL_\bZ$ is large. Recall that a local system $\cL_\bZ$ on an algebraic variety $X$ is called large if its pullback ${\cL_\bZ}_{|Y} := f^\ast \cL_\bZ$ by any non-constant algebraic map $f : Y \arrow X$ is non-trivial \cite{Kollar-shafa}. 

\begin{prop}\label{VPHS_verify_the_assumptions}
A local system that underlies a $\bC$-VPHS whose associated period map has discrete fibres is large.
\end{prop}
\begin{proof}
Indeed, a $\bC$-VPHS on an algebraic variety is trivial if (and only if) the underlying complex local system is trivial, cf. \cite[Theorem 7.24]{Schmid73}.
\end{proof}

\subsection{Reduction to the case where $L$ is ample}
Let $\bar X$, $L$ and $\cL_\bZ$ as in Theorem \ref{increasing positivity}. By noetherianity, one has $ \mathbf{B}_+(L) = \bigcap \limits_j \,  \Supp(D_j) $ for a finite number of effective divisors $D_j \subset \bar X$ such that $A_j := L^{\otimes n_j}(-D_j)$ is ample for an integer $n_j >0$.
Therefore, given an integer $n$ and a closed irreducible subvariety $\bar Z \subset \bar X(n)$ whose projection is not contained in $\mathbf{B}_+(L)$, one has $\pi_n (\bar Z) \not \subset  D_j$ for some $j$, so that 
\[ \vol_{\bar X(n) | \bar Z} ((\pi_n ^\ast L)^{\otimes n_j}) \geq \vol_{\bar X(n) | \bar Z} (\pi_n ^\ast A_j) = (\pi_n ^\ast A_j)^{\dim \bar Z} \cdot \bar Z.\]
It follows that the statement of Theorem \ref{increasing positivity} holds for $L$ if it holds for the $A_j$'s.

\subsection{End of the proof of Theorem \ref{increasing positivity}}
From now on we assume that $L$ is ample. We need to prove that given a real $v>0$, for any sufficiently big positive integer $n$, every subvariety $\bar Z$ of $\bar X(n)$ intersecting $X(n)$ satisfies $L^{\dim \bar Z} \cdot \bar Z \geq v $. Equivalently, we need to prove that this holds for every normal cycle $\bar Z$ on $\bar X(n)$ whose image intersects $X(n)$.\\

For every positive integer $n$, we denote by $\delta(X, n) (= \delta(X, \cL_\bZ, n))$ the degree of the map $X(n) \arrow X$ restricted to a connected component of $X(n)$. This is a well-defined number since the cover $X(n) \arrow X$ is Galois. If $ Y \arrow X$ is normal cycle on $X$, we denote by $\delta(Y,n)$ the positive integer $\delta(Y, {\cL_\bZ}_{|Y}, n)$.\\

For any normal cycle $Y$ on $ X$, the function $n \mapsto \delta(Y,n)$ goes to infinity with $n$, and this holds uniformly as follows.

\begin{lem}\label{Shafarefades}
For any real $v >0$, if $H(v)$ denotes the set of normal cycles $Y \arrow X$ on $X$ of the form $Y := w^{-1}(X)$ for a normal cycle $w : \bar Y \arrow \bar X$ on $\bar X$ such that $L^{\dim \bar Y} \cdot \bar Y \leq v $, then the function 
\[n \mapsto \min \{ \delta(Y,n) \, | \, (Y \arrow X) \in H(v)  \} \]
 goes to infinity with $n$.
\end{lem}

Assuming the lemma for a moment, let us finish the proof of Theorem \ref{increasing positivity}. Fix $v >0$.
By lemma \ref{Shafarefades} there exists an integer $N_v$ such that $\delta(Y,n)   \geq v$ for every integer $n \geq N_v$ and every normal cycle $Y \arrow X$ on $X$ of the form $Y := w^{-1}(X)$ for a normal cycle $w : \bar Y \arrow \bar X$ on $\bar X$ such that $L^{\dim \bar Y} \cdot \bar Y \leq v $.\\
 
Let $n \geq N_v$ and $\bar Z \arrow \bar X(n)$ be a normal cycle on $\bar X(n)$ whose image intersects $X(n)$, and let us prove that $L^{\dim \bar Z} \cdot \bar Z \geq v $. Let $\bar Y$ be the normalization of the image of $\bar Z$ in $\bar X$, so that $\bar Z$ is an irreducible component of $\bar Y(n)$. By the projection formula we have the equality 
\[ L^{\dim \bar Z} \cdot \bar Z = \delta(Y,n) \cdot (L^{\dim \bar Y} \cdot \bar Y).\]

Therefore, if $\bar Z$ satisfies $L^{\dim \bar Z} \cdot \bar Z \leq v$, then $L^{\dim \bar Y} \cdot \bar Y \leq v$, which implies $\delta(Y,n)  \geq v$ by the definition of $N_v$, and so finally $L^{\dim \bar Z} \cdot \bar Z \geq v$ again by the preceding equality since $L^{\dim \bar Y} \cdot \bar Y \geq 1$ by ampleness of $L$. This concludes the proof.

\begin{proof}[Proof of lemma \ref{Shafarefades}]
We first prove that for a single normal cycle $Y \arrow X$ the function $n \mapsto \delta(Y,n)$ goes to infinity with $n$. Choosing a point $y \in Y$ and an isomorphism of $\bZ$-modules $\cL_y \simeq \bZ^r$, we get a monodromy representation $\pi_1(Y,y) \arrow \Gl(r, \bZ)$ with image $\Gamma \subset \Gl(r, \bZ)$. A different choice yields another monodromy subgroup $\Gamma' \subset \Gl(r, \bZ)$ which is conjugated to $\Gamma$ by an element of $\Gl(r, \bZ)$. For any integer $n$, the positive integer $ \delta(Y,n)$ can be reinterpreted as the cardinal of the image of $\Gamma $ by the morphism $\Gl(r, \bZ) \arrow \Gl(r, \bZ / n \bZ)$, therefore it tends to infinity with $n$ (note that $\Gamma$ is infinite since $\cL_\bZ$ is large).\\

The full statement of lemma \ref{Shafarefades} follows similarly, since by proposition \ref{finiteness monodromy representations} there are, up to conjugation by an element of $ \Gl(r, \bZ)$, only finitely many subgroups of $ \Gl(r, \bZ)$ obtained as the monodromy group of the local system induced by $\cL$ on $w^{-1}(X)$ for a normal cycle $w : \bar W \arrow \bar X$ such that $L^{\dim \bar W} \cdot \bar W \leq v $.
\end{proof}

\section{Ramification at infinity}

\begin{theorem}\label{Ramification at infinity}
Let $\bar X$ be a normal compact analytic space, and let $\cL$ be a $\bZ$-local system defined on the complementary $X$ of a divisor $D$. Assume that the monodromies of $\cL$ around the irreducible components of $D$ are infinite and have eigenvalues of modulus 1.\\

For every prime number $p$, let $X(p)$ be the finite \'{e}tale cover of $X$ that trivializes the local system $\cL \otimes_{\bZ} \bF_p$, and let $\bar X(p)$ be the normalization of $\bar X$ in the total ring of fraction of $X(p)$.\\

Then, for almost all prime numbers $p$, the map $\pi_p : \bar X(p) \arrow \bar X$ ramifies over every irreducible component of $D$ with an order divisible by $p$.
\end{theorem}
\begin{proof}
If $p$ is a sufficiently big prime number, we need to prove that for every disk $(\Delta , 0) \subset (\bar X, D)$ transverse to an irreducible component of $D$, the restriction of $\pi_p$ to $\Delta^\ast := \Delta - \{0\} $ takes place in a commutative diagram of holomorphic maps:

 \[\xymatrix{
 \pi_p^{-1}(\Delta^\ast) \ar[d]_{{\pi_p}_{|  \pi_p^{-1}(\Delta^\ast)}} \ar@{..>}[r]  & \Delta^\ast \ar[dl]^{ z \mapsto z^p}\\
\Delta^\ast  & 
}\]

Equivalently, we need to prove that $p$ divides the degree of the connected \' etale covers of $\Delta^\ast$ obtained by restricting $ {\pi_p}_{|  \pi_p^{-1}(\Delta^\ast)} : \pi_p^{-1}(\Delta^\ast) \arrow \Delta^\ast$
to any connected component of $ \pi_p^{-1}(\Delta^\ast)$. Since the pullback of $\cL \otimes_{\bZ} \bF_p$ to $ \pi_p^{-1}(\Delta^\ast)$ is trivial, it is sufficient to prove that the order of the corresponding monodromy matrix is divisible by $p$. Therefore, choosing an isomorphism ${\cL}_{\frac{1}{2}} \simeq \bZ^n$ and letting $T \in \Gl(n, \bZ)$ be the corresponding monodromy matrix of $\cL$ restricted to $\Delta^\ast$, we need to prove that for almost all prime numbers $p$ the order of the matrix $(T \mod p) \in  \Gl(n, \bF_p)$ is divisible by $p$. Since by assumptions the eigenvalues of $T$ have modulus one and $T \in \Gl(n, \bZ)$, it follows from Kronecker's theorem that the eigenvalues of $T$ are some roots of unity and we are done thanks to the following lemma.
\end{proof}
\begin{lem}\label{order of quasiunipotent matrix mod p}
If $T \in \Gl(n, \bZ)$ is a quasi-unipotent matrix of infinite order, then for almost every prime number $p$ the order of the matrix $(T \mod p) \in  \Gl(n, \bF_p)$ is divisible by $p$.
\end{lem}
\begin{proof}[Proof of lemma \ref{order of quasiunipotent matrix mod p}]
Let $T = U \cdot S$ be the Jordan-Chevalley decomposition of $T$ in $\Gl(n, \bQ)$, so that $U$ and $S$ are two commuting matrices in $\Gl(n, \bQ)$ with $U$ unipotent and $S$ semisimple. Since $T$ is quasi-unipotent, the matrix $S$ is of finite order. Note that this implies that $U$ is different from the identity since $T$ has infinite order by assumption.\\ 

Let $p$ be a prime number. If $p$ does not divide any of the denominator of the coefficients of $U$ or $S$, then both $U$ and $S$ can be viewed as elements of $ \Gl(n, \bZ_p)$ and the factorization $T = U \cdot S$ holds in $ \Gl(n, \bZ_p)$. If moreover $p$ does not divide any of the coefficients of $U - I_n$, then the order of $ (U \mod p)$ equals $p^k$ for an integer $k \geq 1$ thanks to the lemma \ref{unipotent mod p} below. If finally $p$ does not divide the order $m$ of $S$, then $(S \mod p)$ is also of order $m$ by the the lemma \ref{finite order mod p}. Therefore, for $p$ big enough, the matrices $(U \mod p)$ and $(S \mod p)$ are commuting and of coprime orders respectively $p^k$ and $m$, hence their product $(T \mod p)$ is of order $p^k \cdot m$. In particular $p$ divides the order of $(T \mod p)$ if $p$ is sufficiently big.
 \end{proof}
 
\begin{lem}\label{unipotent mod p}
If $U \in \Gl(n, \bF_p)$ is a unipotent element, then its order is a power of $p$.
\end{lem}
\begin{proof}
Let $N = U - I_n \in \mathrm{Mat}(n, \bF_p)$, so that by assumption $N^n = 0$. For any integer $k \geq 1$ we have the equality:
\[ U^{p^k} = (Id + N)^{p^k} = I_n + {p^k\choose 1} \cdot N+ {p^k\choose 2} \cdot N^2 +  \ldots  + {p^k\choose n-1} \cdot N^{n- 1} .\]
For $k$ bigger than the $p$-valuation of $(n-1)!$, $p$ divides all the ${p^k \choose i}$ for $2 \leq i \leq n-1$, so that $ U^{p^k} = I_n$ and the order of $U$ divides $p^k$.\\
\end{proof}

\begin{lem}\label{finite order mod p}
If $p$ is any prime number and $S \in \Gl(n, \bZ_p)$ is of finite order $m$ with $p \nmid m$, then $(S \mod p) \in  \Gl(n, \bF_p)$ is of order $m$ too.
\end{lem}
\begin{proof}
Let $S \in \Gl(n, \bZ_p)$ of finite order $m$ with $p \nmid m$. Assume by contradiction that $(S^k \mod p) = I_n \in  \Gl(n, \bF_p)$ for some $k < m$. It follows that a well-chosen power $S^\prime$ of $S$ is of order $l$ for a prime number $l \neq p$ and satisfies $(S^\prime \mod p) = I_n \in  \Gl(n, \bF_p)$. From the equality $(S^\prime)^l = I_n$ we get that 
$(S^\prime - I_n) \cdot (I_n + S^\prime + \ldots + (S^\prime)^{l-1}) = 0$ in $\Gl(n, \bZ_p)$. Since $I_n + S^\prime + \ldots + (S^\prime)^{l-1}$ reduces modulo $p$ to the invertible matrix $ l \cdot I_n$, it is in fact invertible in $ \Gl(n, \bZ_p)$. It follows that $S^\prime = I_n$, a contradiction.
\end{proof}

\section{Proof of Theorem \ref{main result} and Corollary \ref{increasing gonality}}

Let  $\cL_\bZ$ be a torsion-free $\bZ$-local system on a complex algebraic variety $X$. Assume that $\cL_\bC$ underlies a variation of complex polarized Hodge structures $\bV = (\cL_\bC ,  \cF^{\sbt}, h)$ whose period map has discrete fibres. \\

We first prove Theorem \ref{main result}. Fix $v > 0 $. We need to prove that for any sufficiently big prime number $p$, every integral subvariety $Y$ of $X(p)$ satisfies $\vol(Y) \geq v$. \\

Clearly it is sufficient to treat the case where $X$ is integral itself. Let us check by induction on the dimension of $X$ that it is sufficient to prove the result when $X$ is an irreducible smooth affine variety. There is nothing to prove when $\dim X = 0$. If $\dim X >0$, let $H \subset X$ be a divisor containing the singular locus of $X$ and whose complementary $X-H$ is affine. If $p$ is a prime number and $Y$ is an irreducible subvariety of $X(p)$, then 
\begin{itemize}
\item either the projection of $Y$ is contained $H$, so that $Y \subset H^\prime(p)$ for an irreducible component $H^\prime$ of $H$, and we are done by induction; 
\item either the projection of $Y$ intersects $X - H$, so that a Zariski-dense open subset $Y ^o$ of $Y$ is contained in $(X-H)(p)$, and we conclude using that $\vol(Y^o) = \vol(Y)$. \\
\end{itemize}

Assume from now on that $X$ is an irreducible smooth affine variety. Fix a smooth projective variety $\bar X$ and a normal crossing divisor $D \subset \bar X$ such that $\bar X - D$ is identified with $X$. Let $\bar L_{\bV}$ be the associated Griffiths parabolic line bundle on $(\bar X, D)$. Since by assumption the associated period map has discrete fibres, ${\bar L}_{\bV}$ is ample modulo $D$, i.e. $\mathbf{B}_+({\bar L}_{\bV}) \subset D$, cf. Proposition \ref{Griffiths_line_bundle_big}. Therefore ${\bar L}_{\bV}(- \epsilon \cdot D)$ is also ample modulo $D$ for any sufficiently small $\epsilon >0$. Fix such an $\epsilon$. Then, thanks to Corollary \ref{increasing positivity-bis} and Theorem \ref{Ramification at infinity}, for every sufficiently big prime number $p$:
\begin{enumerate}
\item every subvariety $\bar Z$ of $\bar X(p)$ whose projection is not contained in $D = \mathbf{B}_+({\bar L}_{\bV}(- \epsilon \cdot D))$ satisfies $ \vol((\pi_p ^\ast ({\bar L}_{\bV}(- \epsilon \cdot D)))_ {|\bar Z})  \geq v $,
\item the map $\pi_p : \bar X(p) \arrow \bar X$ ramifies over every irreducible component of $D$ with an order divisible by $p$, so that 
\[  \pi_p ^\ast ({\bar L}_{\bV}(- \epsilon \cdot D)) \otimes \cO(E) = (\pi_p ^\ast {\bar L}_{\bV})(- p \cdot \epsilon \cdot D)   \]
for an effective divisor $E \subset \bar X(p)$ supported on $D$.
\end{enumerate}
From now on, $p$ is a prime number such that the two preceding properties are satisfied. Let $\bar Z$ be a closed irreducible subvariety of $\bar X(p)$ whose projection is not contained in $D$. Let $\tilde Z \arrow \bar Z$ be a desingularization such that the set-theoretic preimage $D_{\tilde Z}$ of $D$ by the composition of the maps $\tilde Z \arrow \bar Z \arrow \bar X(p) \arrow \bar X$ is a normal crossing divisor. From $(2)$ and using that $((\pi_p ^\ast D)_ {|\tilde Z})_{red} = D_ {\tilde Z}$, we get the inequality

\[   (\pi_p ^\ast (\bar L_{\bV}(- \epsilon \cdot D)))_ {|\tilde Z} \leq (\bar L_{\bV})_{|\tilde Z}(- p \cdot \epsilon \cdot D_{\tilde Z})  . \]

Taking volumes of both sides and using $(1)$, we get that 
\[   v \leq \vol (( \bar L_{\bV})_{|\tilde Z}(- p \cdot \epsilon \cdot D_{\tilde Z}) ) . \]

On the other hand, applying Theorem \ref{Arakelov inequality} to the $\bC$-VPHS on $\tilde Z - D_{\tilde Z}$ obtained by pullback from $\bV$, and setting $C = \frac{w^2 \cdot \rk{\cL} }{4}$, we get the inequality 
\[ (\bar L_{\bV})_{| \tilde Z} \leq   C \cdot \omega_{\tilde Z}(D_{\tilde Z}), \]
from which it follows that 
\[ \vol( (\bar L_{\bV})_{|\tilde Z} (- C \cdot D_{\tilde Z})   \leq  \vol(  C \cdot \omega_{\tilde Z})   = C^{\dim Z} \cdot \vol(Z). \]

Finally, putting everything together, we get that $v \leq C^{\dim Z} \cdot \vol(Z)$ as soon as $p \cdot \epsilon \geq C$, and this finishes the proof.\\

We now turn to the proof of Corollary \ref{increasing gonality}. Fix an integer $d >0$. For simplicity, we can reduce as before to the case where $X$ is a smooth irreducible affine variety. Let $p_i : X^d \arrow X$ denote the projection on the i-th factor, $S^d X$ denote the d-th symmetric power of $X$ and $\pi : X^d \arrow S^d X$ be the canonical map which realizes $S^dX$ as the quotient of $X^d$ by the natural action of the symmetric group $\mathfrak{S}_d$. Let also $(X^d)^o \subset X^d$ be the complementary of the small diagonals and $(S^d X)^o$ its image in $S^d X$. Then the induced finite \' etale map $\pi : (X^d)^o \arrow (S^d X)^o$ is Galois for the group $\mathfrak{S}_d$. There is a unique torsion-free $\bZ$-local system $\cM$ on $(S^d X)^o$ such that $\pi^\ast \cM \simeq \oplus_{i= 1}^d  p_i^\ast \cL_\bZ$ as $\bZ$-local systems on $(X^d)^o$. (The stalk $\cM_P$ at the point $P = \{x_1, \ldots, x_l\} \in (S^d X)^o$ is equal to $ \oplus_{i= 1}^d (\cL_\bZ)_{x_i}$.) Moreover, $\cM_\bC$ underlies by construction a $\bC$-VPHS whose associated period map has discrete fibres, so that we are in position to apply Theorem \ref{main result} to $(S^d X)^o$ and $\cM$. Fix a prime number $p$ and let $C \hookrightarrow X(p)$ be a curve. A dominant rational map $C \dasharrow \bP^1$ of degree d induces an non-constant rational map $\bP^1 \dashrightarrow S^d (X(p))$ that factorizes through $(S^d X)^o(p)$ (in the last expression the $p$-level structure is taken with respect to $\cM$). But we know from Theorem \ref{main result} that $(S^dX)^o(p)$ cannot contain a curve of genus zero if $p$ is big enough. This finishes the proof.

\section{Applications}\label{Applications}
\subsection{Application to moduli of abelian varieties and curves}
Given two positive integers $g$ and $n$, let $A_g(n)$ denote the coarse moduli space of principally polarized abelian varieties of dimension $g$ with a symplectic\footnote{\label{footnote} Considering symplectic level structures yields a connected component of the moduli space with a level structure as in the rest of the text} level-$n$ structure, and $\bV$ be the variation of Hodge structures coming from the relative cohomology in degree $1$ of the universal family (when $n \leq 2$ the variation $\bV$ is only defined on a dense Zariski-open subset of $A_g(n)$). If $\bar{A_g}(n)$ denotes the first Voronoi compactification of $A_g(n)$ and $D$ the reduced boundary divisor, then Shepherd-Barron {\cite[Theorem 4.1]{Shepherd-Barron}} proved that for any positive rational number $a$ the $\bQ$-line bundle $L_{\bV}(- a \cdot D)$ is ample exactly when $n > 12 \cdot a$ (see \cite[Section 4]{Bruni_AENS} for more details). Applying Corollary \ref{criterion_general_type}, we obtain yet another proof of the following result first proved in \cite[Theorem 1.4]{Bruni_AENS} and \cite[Theorem 5]{cadorel2018subvarieties}.

\begin{theorem}\label{General_type_A_g}
For any $g \geq 1$ and any $n >  6 \cdot g$, every subvariety of $A_g(n)$ is of general type. 
\end{theorem}  

We now deduce some new corollaries.
\begin{cor}\label{Minimal_Gonality_A_g}
Any curve in $A_g(n)$ has gonality at least $ \lceil \frac{n}{6 g} \rceil $.
\end{cor}

\begin{proof}
This is a direct application of Weil descent: if $K \subset L$ is a degree $d$ extension of characteristic zero function fields, then the Weil restriction of a principally polarized abelian variety of dimension $g$ with a level-$n$ structure is a principally polarized abelian variety of dimension $d \cdot g$ with a level-$n$ structure. Therefore, a curve in $A_g(n)$ with gonality $d$ yields a rational curve in $A_{d \cdot g}(n)$ and the proof follows.
\end{proof}

A fortiori, since by  the  Brill-Noether  theorem the gonality  of a smooth projective curve $C$ satisfies $\mathrm{gon}(C) \leq \left\lfloor \frac{g(C)+1}{2} \right\rfloor$ (with equality  for $C$ generic of genus $g(C)$), we obtain a very short proof of the following effective version of results of Noguchi \cite[Main Theorem 1]{Noguchi91} and Hwang-To \cite[Theorem 1.3]{Hwang-To-uniform_boundedness}:
 \begin{cor}\label{Minimal_Genus_A_g}
Any curve in $A_g(n)$ has genus at least $ \lceil \frac{n}{3 g} \rceil -1 $.
\end{cor}

\begin{rem}
With a similar strategy, it is possible to obtain effective results analogous to Theorem \ref{General_type_A_g} and Corollaries \ref{Minimal_Gonality_A_g} and \ref{Minimal_Genus_A_g} for any Shimura variety.
\end{rem}

Thanks to the Torelli embedding, we get the following immediate corollary.
\begin{cor}
If $M_g(n)$ denotes the coarse moduli space of genus g curves with a level-n structures (on their Jacobians), then any curve in $M_g(n)$ has gonality at least $ \lceil \frac{n}{6 g} \rceil $ and genus at least $ \lceil \frac{n}{3 g} \rceil -1 $.
\end{cor}

\subsection{Application to moduli spaces of polarized Calabi-Yau varieties}
Consider a smooth projective morphism $f : \mathcal{X} \arrow \cM$ between complex algebraic varieties (or more generally between separated finite type Deligne-Mumford stacks). Fix a positive integer $k$ and assume that the period map associated to the variation of Hodge structures on the $\bZ$-local system $\cL := R^kf_\ast \bZ \slash (\text{torsion})$ is quasi-finite. By a result of Griffiths \cite{GriffithsIII}, this is for example the case when the fibres of $f$ have a trivial canonical bundle. For any prime number $p$, we denote by $\cM(p)$ the finite \'etale cover of $\cM$ constructed from $\cL \otimes_{\bZ} \bF_p$ as in the introduction. Then applying Theorem \ref{main result} and Corollary \ref{increasing gonality}, we get 
\begin{theorem}\label{inf_Torelli}
The minimal volume (resp. gonality) of an integral subvariety (resp. curve) of $\cM(p)$ tends to infinity with the prime number $p$. 
\end{theorem}

This applies in particular to $\cM(p)$ the moduli stack of polarized Calabi-Yau varieties equipped with a level-p structure, where by definition a level-p structure on a smooth projective complex variety $X$ of dimension $d$ is a basis of the $\bF_p$-vector space $\HH^d(X, \bF_p)$.

\bibliographystyle{alpha}
\bibliography{biblio-VHS}

\end{document}